\documentclass{amsart}

\usepackage[T1]{fontenc}
\usepackage{amsmath}
\usepackage{amssymb}
\usepackage{amsthm}
\usepackage{epsfig}
\usepackage{subfig}
\input xy
\xyoption{all}
\usepackage[all]{xy} 

\newtheorem{theorem}{Theorem}[section]
\newtheorem{corollary}[theorem]{Corollary}
\newtheorem{proposition}[theorem]{Proposition}
\theoremstyle{definition}
\newtheorem{definition}[theorem]{Definition}
\numberwithin{equation}{section}

\newcommand{\R}{\mathcal{R}}
\newcommand{\Ra}[2]{\R_{#1,#2}}
\newcommand{\C}{\mathcal{C}}
\newcommand{\Ca}[2]{\C_{#1,#2}}
\newcommand{\Pc}{\mathcal{P}}
\newcommand{\Sc}[1]{\mathcal{S}_{#1}}
\newcommand{\T}{\mathcal{T}}
\newcommand{\Ta}[2]{\T_{#1,#2}}

\DeclareMathOperator{\gen}{{\it gen}}
\DeclareMathOperator{\lon}{{\it lon}}

\begin{document}
\title[The Topology of Knight's Tours on Surfaces]
{The Topology of Knight's Tours on Surfaces}

\author{Bradley Forrest}
\address{101 Vera King Farris Drive; Galloway, NJ 08205}
\email{bradley.forrest@stockton.edu}

\author{Kara Teehan}
\address{101 Vera King Farris Drive; Galloway, NJ 08205} 
\email{teehank@go.stockton.edu}

\subjclass[2010]{Primary 05C45, Secondary 57M15, 57M10, 57M05}

\keywords{Knight's Tour, Surfaces, Fundamental Groups} 

\begin{abstract}
We investigate the homotopy classes of closed knight's tours on cylinders and tori.  Specifically, we characterize the dimensions of cylindrical chessboards that admit closed knight's tours realizing the identity of the fundamental group and those that admit closed tours realizing a generator of the fundamental group.  We also produce analogous results for toroidal chessboards.
\end{abstract}

\maketitle

\section{\bf Introduction}

The knight's tour problem is a mainstay of recreational mathematics and a classical problem in graph theory.  One source of the problem's intrigue comes from its namesake chess piece, the knight, which moves in an ``L'' shaped pattern two squares either vertically or horizontally followed by one square in a perpendicular direction.  A \emph{knight's tour} is a sequence of moves in which a knight visits each square on a chessboard exactly once.  A tour is \emph{closed} if the knight can, in a single move, return to the starting square from the ending square; otherwise the tour is \emph{open}.  The traditional knight's tour problem is to find a closed knight's tour on a standard $8 \times 8$ chessboard.

Solutions to the traditional knight's tour problem have been known for centuries.  Cull and Decurtains generalized the problem to boards with dimensions $m \times n$ where $m, n \geq 5$.  Specifically, they proved that each such board supports an open tour and characterized which boards of this class support a closed tour \cite{cull-decurtins}.  In 1991, Schwenk completed this characterization, determining for each rectangular board whether or not the board admits a closed tour \cite{schwenk}.  Since that time there has been significant investigation of a variety of generalizations of the closed tour problem.  Watkins studied closed tours on cylinders and, together with Hoenigman, investigated closed tours on tori \cite{watkins}, \cite{watkins-hoenigman}.  Miller and Farnsworth recently extended this study by investigating closed knight's tours on both cylinders and tori with one square removed \cite{miller-farnsworth}.

Many researchers have investigated the knight's tour problem on higher dimensional objects, most notably rectangular prisms, but also objects of dimension greater than 3  \cite{demaio}, \cite{demaio-mathew}, \cite{erde}, \cite{erde-golenia-golenia}, \cite{yang-zhu-jiang-huang}.  There have also been investigations of generalized knight moves in which a knight moves in an $a$ by $b$ ``L'' shape instead of the standard $1$ by $2$ ``L'' shape \cite{chia-ong}, \cite{yang-zhu-jiang-huang}.

In this work, we explore knight's tours on surfaces, specifically $m \times n$ cylinders $C$, cylinders of length $m$ and circumference $n$, and $m \times n$ tori $T$, tori with latitudinal circumference $m$ and longitudinal circumference $n$.  For each of these boards we construct a graph where there is a vertex for each square on the board and one edge between a pair of vertices for each possible move that takes the knight from one of those squares to the other.  While this construction typically makes a graph, for some surfaces with small dimensions this can result in a pseudograph or a multigraph.  For example, the multigraph associated to the $2 \times 1$ cylinder consists of 2 vertices which are connected by 2 edges.  

These graphs map naturally to the $m \times n$ cylinder and $m \times n$ torus.  After choosing the image of a vertex in each surface to act as the base point, $c$ and $t$ in $C$ and $T$ respectively, directed closed tours in the graphs determine elements of $\pi_1(C,c)$ and $\pi_1(T,t)$.  In this work, for each of the following four conditions, we characterize the values of $m$ and $n$ that satisfy the condition.
\begin{itemize}
\item There exists a tour that realizes the identity of $\pi_1(C,c)$.
\item There exists a tour that realizes a generator of $\pi_1(C,c)$.
\item There exists a tour that realizes the identity of $\pi_1(T,t)$.
\item There exists a tour that realizes the homotopy class of a longitude in $\pi_1(T,t)$.
\end{itemize}
These characterizations are given in Theorems \ref{thm:cylnull}, \ref{thm:cylgen}, \ref{cor:ToriNull},  and \ref{thm:torusgen} respectively.
These topological questions can, in part, be answered by work on open tours on regular rectangular boards.  Specifically, open tours where the tour ends at a square from which, if the top and bottom of the board were identified, the knight could move to the starting square of the tour realize a generator of $\pi_1(C,c)$ and realize the homotopy class of a longitude in $\pi_1(T,t)$.  Significant work has been completed studying open tours on regular boards.  In particular, Cannon and Dolan showed that on rectangular chessboards where the product of $m$ and $n$ is even and both $m$ and $n$ are greater than or equal to six, given any pair of squares of opposite colors there exists an open tour starting at one of the squares and ending at the other \cite{cannon-dolan}. Ralston showed that all boards where $m$ and $n$ are both odd with $m \geq 7$ and $n \geq 5$ are ``odd-tourable'', meaning that for every pair of squares with the same color as the corner squares there is a knight's tour that begins on one square of the pair and ends on the other \cite{ralston}.

This paper is organized as follows.  Sections 2 and 3 present the required background information in Graph Theory and Algebraic Topology respectively.  We characterize the values of $m$ and $n$ for which the $m \times n$ cylinder and $m \times n$ torus admit tours that realize the identities of their respective fundamental groups in Section 4.  Section 5 characterizes the dimensions of cylinders that admit a tour that realizes a generator of the fundamental group, while Section 6 characterizes the dimensions of tori that admit a tour that realizes the homotopy class of a longitude.  In Section 7, we discuss potential future work.

\section{\bf Graph Theory Background}

In this section, we review relevant theorems of Schwenk and Watkins.  We also set our notation beginning with \emph{regular} $m \times n$ \emph{ boards} which are rectangular boards with $m$ vertical columns and $n$ horizontal rows.  Specifically, we label each square by an ordered pair of integers $(a,b)$ where $0 \leq a \leq m-1$ and $0 \leq b \leq n-1$.  A \emph{knight pair} is an ordered pair of integers $(x,y)$ so that $|x| = 2$ and $|y| = 1$ or vice versa.  A \emph{regular jump} is a knight pair $(x,y)$ together with a position on the board $(a,b)$ such that $(a+x, b+y)$ is also a position on the board.  Two regular jumps $(x_1,y_1), (a_1, b_1)$ and $(x_2, y_2), (a_2,b_2)$ are equivalent if $(x_1, y_1)~=~(-x_2, -y_2)$ and $(a_1 + x_1, b_1+y_1) = (a_2, b_2)$.  Equivalence classes of regular jumps are called \emph{regular moves}, and more directly there is a regular move incident to $(a_1, b_1)$ and $(a_2, b_2)$ if $(a_1 - a_2, b_1 - b_2)$ is a knight pair.  The graph with a vertex for each position on an $m \times n$ regular board and an edge for each regular move is denoted $\Ra{m}{n}$.  

We will usually specify an edge by listing the two vertices to which the edge is incident.  On $\Ra{m}{n}$ this causes no ambiguity.  However, we will abuse notation slightly and use this convention for the multigraphs modeling the cylinder and torus.  This is typically sufficient but when there is more than one edge incident to a pair of vertices and that choice of edge is relevant to our discussion, we will specify the edge by giving a vertex and a knight pair.

Our work focuses on closed knight's tours, which, for the remainder of this paper, we simply refer to as knight's tours.  Extending the work of Cull and Decurtins, Schwenk characterized which regular $m \times n$ boards admit knight's tours.

\begin{theorem}[Schwenk \cite{schwenk}]
The graph $\Ra{m}{n}$, with $m \geq n$ and at least one of $m$ and $n$ greater than 1, admits a Hamiltonian tour if and only if 
\begin{itemize}
\item $m$ and $n$ are not simultaneously odd,
\item $n=1$, $2$, or $4$, or
\item $n = 3$ while $m=4$, $6$, or $8$. 
\end{itemize}
\label{thm:schwenk} 
\end{theorem}

Note that $\Ra{m}{n}$ supports a Hamiltonian tour if and only if $\Ra{n}{m}$ also supports a Hamiltonian tour.  For this reason, we will also apply Theorem \ref{thm:schwenk} when discussing $\Ra{m}{n}$ for boards with $n \geq m$.  

We primarily discuss knight's tours on surfaces.  One surface that we focus on is the $m \times n$ \emph{cylinder}.  We construct an $m \times n$ cylinder by identifying the top and bottom borders of a regular $m \times n$ board.  This identification adds knight moves that cross the identified border.  More specifically, a \emph{cylindrical jump} is a knight pair $(x,y)$ together with a position on the regular board $(a,b)$ such that $(a+x, (b+y) \mod n)$ is a position on the board.  Two cylindrical jumps $(x_1,y_1), (a_1, b_1)$ and $(x_2, y_2), (a_2,b_2)$ are equivalent if $(x_1, y_1) = (-x_2, -y_2)$ and $(a_1 + x_1, (b_1+y_1) \mod n) = (a_2, b_2)$.   Equivalence classes of cylindrical jumps that are not regular moves are called \emph{cylindrical moves}.  The multigraph with a vertex for each position on a regular $m \times n$ board and an edge for each regular or cylindrical move is denoted $\Ca{m}{n}$.  Note that $\Ra{m}{n}$ is a canonical subgraph of $\Ca{m}{n}$.

Most research conducted on knight's tours on cylinders has focused on the quotient graph of $\Ca{m}{n}$ given by identifying all edges incident to the same pair of vertices.  We apply these results to $\Ca{m}{n}$ since, with one exception, there exists a Hamiltonian cycle on $\Ca{m}{n}$ if and only if such a cycle exists on the quotient.  The only exception occurs when $m=2$ and $n=1$, when there is a Hamiltonian cycle on $\Ca{2}{1}$ but no such cycle on the quotient.  Watkins characterized the values of $m$ and $n$ for which the quotient admits a Hamiltonian cycle, and we adapt this result to $\Ca{m}{n}$.

\begin{theorem}[Watkins \cite{watkins}; pg 71]
The multigraph $\Ca{m}{n}$ admits a Hamiltonian cycle unless $m=1$ and $n>1$, or $m=2$ or $4$ and $n$ is even.
\label{thm:watkins}
\end{theorem}

In addition to studying knight's tours on cylinders, we will also investigate tours on tori.  We construct an $m \times n$ \emph{torus} by identifying the sides of the $m \times n$ cylinder, which adds new possible knight moves.  More specifically, a \emph{toroidal jump} is a knight pair $(x,y)$ together with a position on the regular board $(a,b)$ such that $((a+x) \mod m, (b+y) \mod n)$ is a position on the board.  We say that two toroidal jumps $(x_1,y_1), (a_1, b_1)$ and $(x_2, y_2), (a_2,b_2)$ are equivalent whenever $(x_1, y_1) = (-x_2, -y_2)$ and $((a_1 + x_1) ~\mod m, (b_1+y_1)\mod n)=(a_2, b_2)$.   Equivalence classes of toroidal jumps that are not regular moves are called \emph{toroidal moves}.  The multigraph with a vertex for each position on a regular $m \times n$ board and an edge for each regular or toroidal move is denoted $\Ta{m}{n}$.  For some small values of $m$ and $n$, this construction gives a pseudograph.  For example, $\Ta{1}{1}$ is a single vertex together with 4 edges.  Note that $\Ca{m}{n}$ is a canonical submultigraph of $\Ta{m}{n}$. 

We make significant use of covering space theory in our study of knight's tours on cylinders and tori.  To that end, we define covering graphs $\Sc{m}$ and $\Pc$ respectively for $\Ca{m}{n}$ and $\Ta{m}{n}$.  These covering graphs model the infinite strip of width $m$ and the plane, the respective universal covers of the cylinder and torus.

\begin{figure}[t]
\subfloat[]{\fbox{\includegraphics[scale = .35]{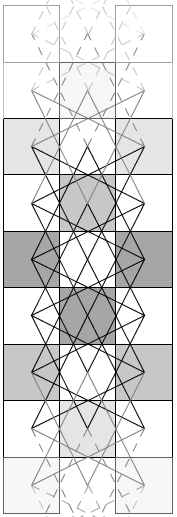}}}
\hfill
\subfloat[]{\fbox{\includegraphics[scale = .35]{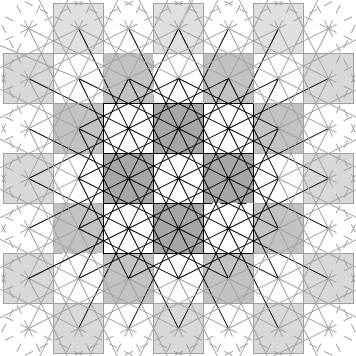}}}
\caption{(A) Subset of $\Sc{3}$ (B) Subset of $\Pc$}
\label{fig:SnPlane}
\end{figure}

The graph $\Sc{m}$ has a vertex for each ordered pair of integers $(a,b)$ where $0 \leq a \leq m-1$ and an edge between $(a_1,b_1)$ and $(a_2, b_2)$ if  $(a_1 - a_2, b_1 - b_2)$ is a knight pair.  The vertex labeling makes $\Sc{l}$ a subgraph of $\Sc{m}$ when $l < m$.  Note that there is a combinatorial map from $\phi_c \colon \Sc{m} \to \Ca{m}{n}$ given by mapping the vertex $(a,b)$ to $(a, b \mod n)$ and the edge $(x,y),(a,b)$ to $(x,y),(a, b \mod n)$.  A subset of $\Sc{3}$ is shown in Frame (A) of Figure \ref{fig:SnPlane}.

The plane is the universal cover of the torus, and to model the plane we build the graph $\Pc$ with a vertex for each ordered pair of integers and an edge between $(a_1,b_1)$ and $(a_2, b_2)$  if  $(a_1 - a_2, b_1 - b_2)$ is a knight pair.   There exists a combinatorial map from $\phi_t \colon \Pc \to \Ta{m}{n}$ given by mapping the vertex $(a,b)$ to $(a \mod m, b \mod n)$ and the edge $(x,y),(a,b)$ to $(x,y),(a \mod m, b \mod n)$.  A subset of $\Pc$ is shown in Frame (B) of Figure \ref{fig:SnPlane}.

Note that $\phi_c \colon \Sc{m} \to \Ca{m}{n}$ and $\phi_t \colon \Pc \to \Ta{m}{n}$ are both covering maps.  One important consequence of this is that edge cycles in $\Ca{m}{n}$ and $\Ta{m}{n}$ lift to edge paths in $\Sc{m}$ and $\Pc$.  More specifically, let $f \colon I \to \Ca{m}{n}$ and $g \colon I \to \Ta{m}{n}$ be edge cycles, where $I$ is the closed unit interval, and let $u$ and $v$ be lifts of $f(0)$ and $g(0)$.  Then there exist unique lifts $\tilde{f} \colon I \to \Sc{m}$ and $\tilde{g} \colon I \to \Pc$ so that $\tilde{f}(0) = u$ and $\tilde{g}(0) = v$.

Frame (A) of Figure \ref{fig:Lift} shows a loop in $\Ca{5}{2}$ where the dotted lines denote cylindrical moves, while Frame (B) is a lift of the loop in Frame (A) to $\Sc{5}$.  Similarly, Frame (C) of Figure \ref{fig:Lift} shows a loop in $\Ta{4}{4}$ where the dotted lines denote toroidal moves, while Frame (D) is a lift of the loop in Frame (C) to $\Pc$. 

Our figures throughout this work show tours on $\Ca{m}{n}$, $\Ta{m}{n}$, $\Sc{m}$, and $\Pc$.  In these figures, we denote the base point $(0,0)$ with a dot.  In $\Ca{m}{n}$ and $\Ta{m}{n}$, this is always the bottom left square in the figure.  We tile $\Sc{m}$ and $\Pc$ by fundamental domains, and show these domains by darkened lines.  We choose the bottom leftmost square in one of these domains to be our base point $(0,0)$.  In each figure we coordinatize each board to correspond with standard cartesian coordinates.

\begin{figure}[t]
\subfloat[]{\fbox{\includegraphics[scale = .25]{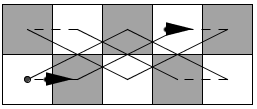}}}
\hfill
\subfloat[]{\fbox{\includegraphics[scale = .25]{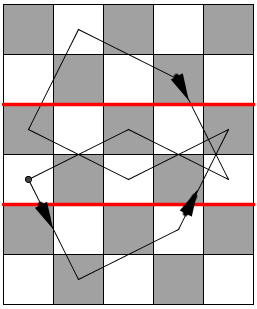}}}
\hfill
\subfloat[]{\fbox{\includegraphics[scale = .25]{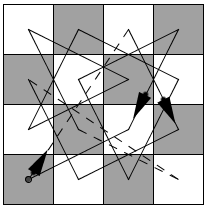}}}
\hfill
\subfloat[]{\fbox{\includegraphics[scale = .25]{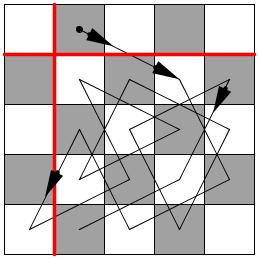}}}
\caption{(A) Tour in $\Ca{5}{2}$  (B) Lift of tour in Frame (A) to $\Sc{5}$  (C) Tour in $\Ta{4}{4}$  (D) Lift of tour in Frame (C) to $\Pc$}
\label{fig:Lift}
\end{figure}

\section{\bf Covering Graphs Background}

In this section, we apply covering space theory to knight's tours.  To do this, we will define standard maps from $\Ca{m}{n}$ to a cylinder $C$ with length $m$ and  circumference $n$ and from $\Ta{m}{n}$ to a torus $T$ with latitudinal  circumference $m$ and longitudinal circumference $n$.  We largely treat $\Ca{m}{n}$ and $\Ta{m}{n}$ as subsets of $C$ and $T$ respectively, and, after base point vertices $c$ and $t$ are chosen, Hamiltonian cycles in $\Ca{m}{n}$ and $\Ta{m}{n}$ determine elements of $\pi_1(C,c)$ and $\pi_1(T,t)$ respectively.

To define our maps $\Ca{m}{n} \to C$ and $\Ta{m}{n} \to T$, we will use the universal covers of $C$ and $T$.  To that end, let $P$ be the cartesian plane and $S \subset P$ be the closed infinite strip of width $m$ bounded by the vertical lines $x = 0$ and $x = m$. We define the map $i_c \colon \Sc{m} \to S$ where $i_c$ maps the vertex $(a,b)$ to $(a +.5, b+.5)$ and maps each edge to the unique geodesic line segment between the images of the vertices to which it is incident.  We define the cylinder $C$ as the quotient of $S$ given by identifying all pairs $(a,b)$ and $(a,b')$ where $b$ and $b'$ are equivalent mod $n$.  We can label the points in $C$ by their unique representative in $[0,m] \times [0, n)$.  With this labeling, the covering map $p_c \colon S \to C$ is given by $p_c(a,b) = (a, b \mod n)$.  Lastly, there exists a map $j_c \colon \Ca{m}{n} \to C$ that makes the following diagram commute.
\[
    \xymatrix{
      \Sc{m} \ar[r]^{\textstyle i_c}\ar[d]_{\textstyle\phi_c} & S\ar[d]^{\textstyle p_c}   \\
    \Ca{m}{n} \ar[r]_{\textstyle j_c} & C
    }
\]
For a vertex $(a,b) \in \Ca{m}{n}$, let $j_c(a,b) = (a+.5, (b+.5)\mod n)$.  Given an edge $(x,y),(a,b)$ of $\Ca{m}{n}$, the image of the edge under $j_c$ is not uniquely defined by the image of the vertices to which it is incident.  However, the preimage under $\phi_c$ consists of all of the edges of the form $(x,y)(a,b')$ where $b' \mod n = b$, and under the composition $p_c \circ i_c$ these edges map to a unique path between $j_c(a,b)$ and $j_c(a+x, (b + y)\mod n)$.  More specifically, $j_c$ maps $(x,y)(a,b)$ to the image of the path $f \colon I \to C$ where the path is given by $f(t) = (a + .5 + xt, (b + .5 + yt) \mod n)$.  Note that $j_c$ is well-defined as $(-x, -y)(a+x, (b+y) \mod n)$ gives the same path.  Frame (A) of Figure \ref{fig:image} shows $j_c(\Ca{3}{3})$ in $C$, where the white dots are the images of the vertices of $\Ca{3}{3}$.

\begin{figure}[t]
\subfloat[]{\fbox{\includegraphics[scale = .17]{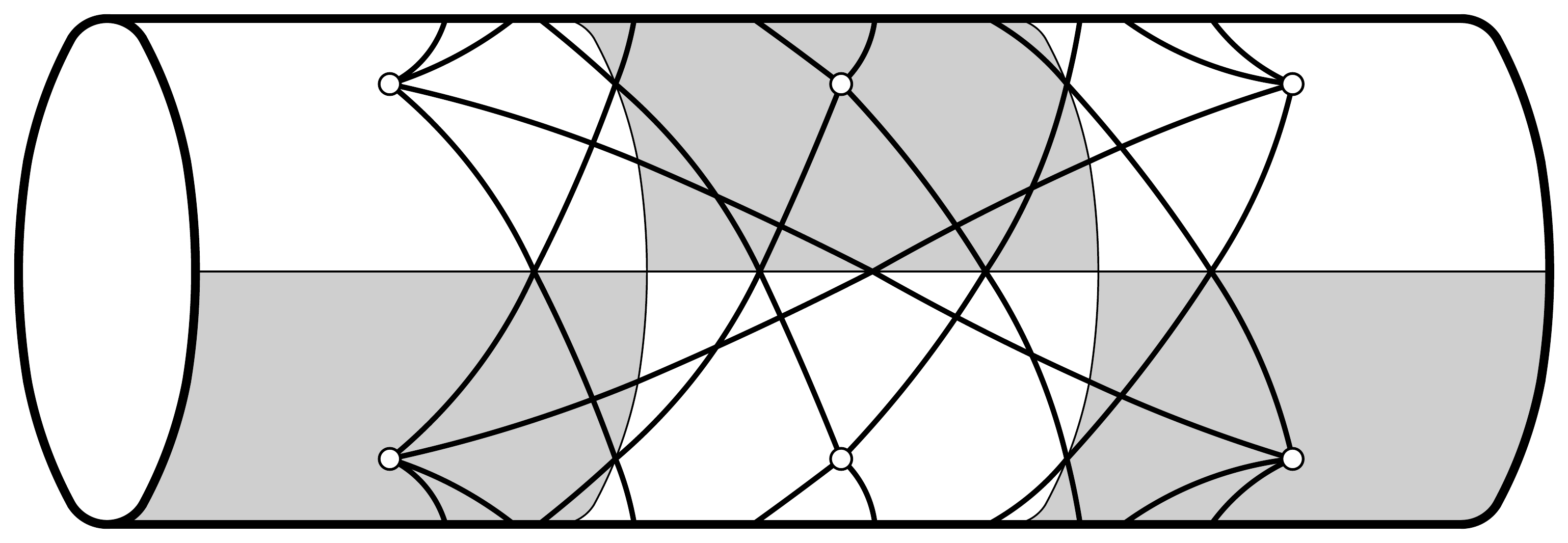}}}
\hfill
\subfloat[]{\fbox{\includegraphics[scale = .15]{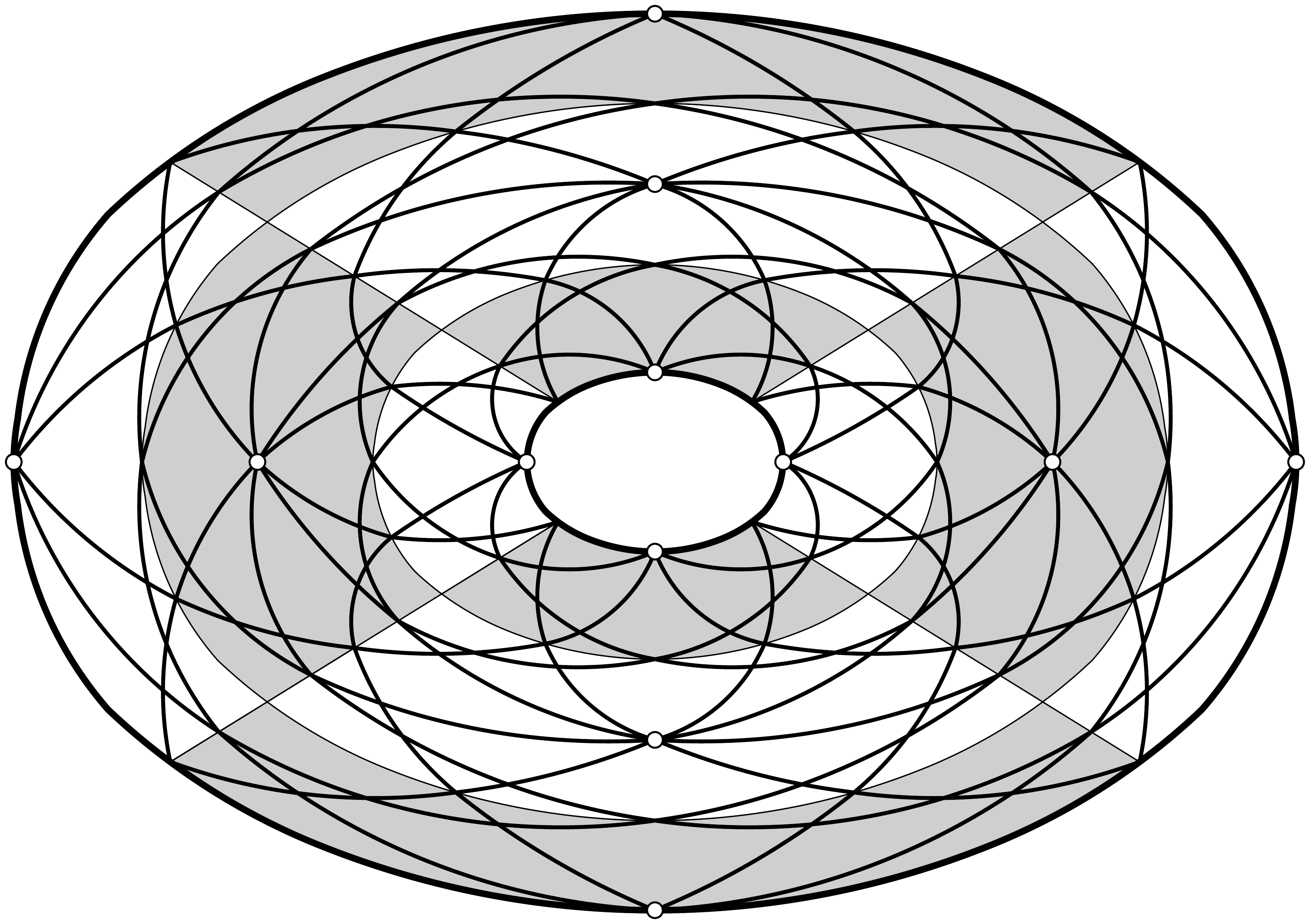}}}
\caption{(A) Image of $\Ca{3}{3}$ in the $3 \times 3$ cylinder  (B) Image of $\Ta{4}{4}$ in the $4 \times 4$ torus}
\label{fig:image}
\end{figure}

Let $i_t \colon \Pc \to P$ be the map that sends the vertex $(a,b)$ to $(a +.5, b+.5)$ and maps each edge to the unique geodesic line segment between the images of the vertices to which it is incident.  We define the torus $T$ as the quotient of $P$ given by identifying all pairs $(a,b)$ and $(a',b')$ where $a$ and $a'$ are equivalent mod $m$ and $b$ and $b'$ are equivalent mod $n$.  We can label the points in $T$ by their unique representative in $[0,m) \times [0, n)$.  With this labeling, the covering map $p_t \colon P \to T$ is given by the formula $p_t(a,b) = (a \mod m, b \mod n)$.  Lastly, there exists a map $j_t \colon \Ta{m}{n} \to T$ that makes the following diagram commute.

\[
    \xymatrix{
      \Pc \ar[r]^{\textstyle i_t}\ar[d]_{\textstyle\phi_t} & P\ar[d]^{\textstyle p_t}   \\
    \Ta{m}{n} \ar[r]_{\textstyle j_t} & T
    }
\]

For vertex $(a,b)$ of $\Ta{m}{n}$, let $j_t(a,b) = ((a + .5) \mod m, (b+.5) \mod n)$. Given an edge $(x,y),(a,b)$ of $\Ta{m}{n}$, the image of the edge under $j_t$ is not uniquely defined by the image of the vertices to which it is incident.   However, the preimage under $\phi_t$ consists of all edges of the form $(x,y)(a',b')$ where $a' \mod m = a$ and $b' \mod n = b$.  Under the composition of functions $p_t \circ i_t$ these edges map to a unique path between the points  $j_t(a,b)$ and $j_t((a+x) \mod m, (b + y) \mod n)$.  More specifically, $j_t$ maps $(x,y)(a,b)$ to the image of the path $f \colon I \to T$ given by $f(t) = ((a + .5 + xt) \mod m, (b + .5 + yt) \mod n)$.  Note that $j_t$ is well-defined as $(-x, -y)((a+x) \mod m, (b + y) \mod n)$ maps to the same path.  Frame (B) of Figure \ref{fig:image} shows $j_c(\Ta{4}{4})$ in $T$, where the white dots are the images of the vertices of $\Ta{4}{4}$.

We are now in a position to apply covering space theory to knight's tours on cylinders and tori.  Recall the following classical theorem:

\begin{theorem}
Let $\tilde{X}$ be the universal cover of $X$ with covering map $p \colon \tilde{X} \to X$, and $x \in X$ with $\tilde{x} \in p^{-1}(x)$.  Let $f, g \colon (I, 0) \to (X, x)$ be two loops with lifts $\tilde{f}, \tilde{g} \colon (I, 0) \to (\tilde{X}, \tilde{x})$ where $I$ is the closed unit interval.  Then $f$ and $g$ are path homotopic if and only if $\tilde{f}(1) = \tilde{g}(1)$.
\label{thm:cover}
\end{theorem}

In our context, Theorem \ref{thm:cover} has important consequences.  Let the loop $f \colon (I,0) \to (\Ca{m}{n},(0,0))$ be a Hamiltonian tour, and consider its image $j_c \circ f$.  The homotopy class of $j_c \circ f$ defines an element of $\pi_1(C, c)$ where $c = j_c(0,0)$.  Similarly, Hamiltonian tours in $\Ta{m}{n}$ define elements of $\pi_1(T, t)$ where $t = j_t(0,0)$.

There are two particular curves in $C$ and $T$ that we will study.  Let $\gen \colon (I,0) \to (C,c)$ be given by $\gen(t) = (.5, (.5 + nt) \mod n)$ and $\gen'$ be $\gen$ traveled in the opposite direction.  Note that the homotopy classes of $\gen$ and $\gen'$ are the generators of $\pi_1(C, c)$.  Also, let $\lon \colon (I,0) \to (T,t)$ given by $\lon(t) = (.5, (.5 + nt) \mod n)$ and $\lon'$ be $\lon$ traveled in the opposite direction.  The loops $\lon$ and $\lon'$ are \emph{the longitudinal loops} of $T$.  We are particularly interested in Hamiltonian tours in $\Ca{m}{n}$ whose images in $C$ are homotopic to $\gen$ or $\gen'$ and Hamiltonian tours in $\Ta{m}{n}$ whose images in $T$ are homotopic to $\lon$ or $\lon'$.

\begin{definition}
Let $f \colon (I,0) \to (\Ca{m}{n}, (0,0))$ and $g \colon (I,0) \to (\Ta{m}{n},(0,0))$ be Hamiltonian tours.  Then $f$ is \emph{nullhomotopic} if the homotopy class of $j_c \circ f$ is the identity in $\pi_1(C,c)$, and $f$ \emph{realizes a generator} if $j_c \circ f$ is homotopic to $\gen$ or $\gen'$.  Further, $g$ is \emph{nullhomotopic} if the homotopy class of $j_t \circ g$ is the identity in $\pi_1(T,t)$, and $g$ \emph{realizes the longitude} if $j_t \circ g$ is homotopic to $\lon$ or $\lon'$.
\end{definition}

Note that $f$ lifts to a path $\tilde{f}$ in $\Sc{m}$, and by Theorem \ref{thm:cover} we can see that the endpoint of this path determines the homotopy class of $j_c \circ f$.  This argument is made precise for both cylinders and tori in the corollary below.

\begin{corollary}
Let $f \colon (I, 0) \to (\Ca{m}{n}, (0,0))$ and $g \colon (I, 0) \to (\Ta{m}{n},(0,0))$ be Hamiltonian tours with lifts $\tilde{f} \colon I \to \Sc{m}$ and $\tilde{g} \colon I \to \Pc$ with given initial points $\tilde{f}(0) = (0,0)$ and $\tilde{g}(0) = (0,0)$.  Then:
\begin{enumerate}
\item  The loop $f$ is nullhomotopic if and only if $\tilde{f}(1) = (0,0)$.
\item The loop $f$ realizes a generator of $\pi_1(C,c)$ if and only if either $\tilde{f}(1)~=~(0,n)$ or $\tilde{f}(1) = (0,-n)$.
\item The loop $g$ is nullhomotopic if and only if $\tilde{g}(1) = (0,0)$.
\item The loop $g$ realizes a longitude if and only if either $\tilde{g}(1) = (0,n)$ or $\tilde{g}(1) = (0,-n)$.
\end{enumerate}
\label{cor:cover}
\end{corollary}

\begin{proof}
Note that $i_c \circ \tilde{f}$ is a lift of $j_c \circ f$ with $i_c \circ \tilde{f}(0) = (.5,.5)$.

Consider statement (1).  The constant map with image $c$ lifts to a constant map with image $(.5,.5)$ and $j_c \circ f$ is nullhomotopic if and only if $i_c \circ \tilde{f}$ is homotopic to this lift of the constant map.  Hence, by Theorem \ref{thm:cover}, the map $j_c \circ f$ is nullhomotopic if and only if $i_c \circ \tilde{f}(1) = (.5, .5)$ which is equivalent to $\tilde{f}(1) = (0,0)$.  An analogous argument proves statement (3).

Consider statement (2).  An argument analogous to the previous paragraph's discussion that uses $\gen$ and $\gen'$ in place of the constant map establishes statement (2). Using this argument with $\lon$ and $\lon'$ proves statement (4).
\end{proof}

Important two colorings on $\Sc{m}$ and $\Pc$ are given by coloring the vertex $(a,b)$ red if $a+b$ is even and blue if $a+b$ is odd.  The graphs $\Sc{m}$ and $\Pc$ are indeed bipartite as the parity of the sum changes when any edge is traversed.

\begin{proposition}
Let $n$ be odd.  If $m$ is odd and at least one of $m$ and $n$ is larger than 1, then there is no nullhomotopic tour on $\Ca{m}{n}$ nor on $\Ta{m}{n}$.  If $m$ is even, then there is no tour on $\Ca{m}{n}$ that realizes a generator and no tour on $\Ta{m}{n}$ that realizes a longitude.
\label{prop:oddeven}
\end{proposition}

\begin{proof}
Let $m$ be odd, and for the sake of contradiction, suppose that $f \colon (I,0) \to (\Ca{m}{n}, (0,0))$ is a nullhomotopic tour with lift $\tilde{f}$.  Then, by Corollary \ref{cor:cover}, $\tilde{f}$ is an edge cycle that traverses an odd number of edges in a bipartite graph, which cannot exist.  An analogous argument proves this statement for $\Ta{m}{n}$.

Let $m$ be even, and for the sake of contradiction, suppose that the loop $f \colon (I,0) \to (\Ca{m}{n}, (0,0))$ realizes a generator and has lift $\tilde{f}$.  Note that, by Corollary \ref{cor:cover}, the starting and ending vertices of $\tilde{f}$ have opposite colors.  Then $\tilde{f}$ is an edge path in a bipartite graph that traverses an even number of edges but has oppositely colored starting and ending vertices, which cannot exist.  An analogous argument proves that there is no tour on $\Ta{m}{n}$ that realizes the longitude.
\end{proof}

\section{\bf Nullhomotopic Tours on Cylinders}

In this section, we characterize the values of $m$ and $n$ for which $\Ca{m}{n}$ and $\Ta{m}{n}$ support nullhomotopic knight's tours.  More specifically, for $\Ca{m}{n}$ we prove:

\begin{theorem}
The multigraph $\Ca{m}{n}$ supports a nullhomotopic tour if and only if none of the following hold:
\begin{itemize}
\item $m$ and $n$ are simultaneously odd and at least one of $m$ and $n$ is greater than 1, 
\item $m = 1$ and $n > 1$,
\item $m = 2$, or
\item $m = 4$ and $n$ is even.
\end{itemize}
\label{thm:cylnull}
\end{theorem}

One consequence of Theorem \ref{thm:cylnull} is that when $m$ and $n$ are both greater than or equal to 5, $\Ca{m}{n}$ admits a nullhomotopic tour if and only if $\Ra{m}{n}$ admits a closed tour.  This is not true for smaller boards; there are many cases when at least one of $m$ or $n$ is less than 5 for which $\Ca{m}{n}$ admits a nullhomotopic tour but $\Ra{m}{n}$ does not admit a closed tour.

We take a case by case approach to prove Theorem \ref{thm:cylnull}.  When $m, n \geq 5$, if at least one of $m$ or $n$ is even, Theorem \ref{thm:schwenk} states that there exists a tour on a regular $m \times n$ board and thus there exists a nullhomotopic tour on a cylinder.  When both $m$ and $n$ are odd, the $m \times n$ cylinder does not support a nullhomotopic tour by Proposition \ref{prop:oddeven}.  We are left to consider $\Ca{m}{n}$ when at least one of $m$ or $n$ is less than 5.  For many board sizes, we will exhibit a nullhomotopic tour by using statement (1) of Corollary \ref{cor:cover} and producing a closed edge path in $\Sc{m}$ that maps to the desired nullhomotopic tour.  See Figures \ref{fig:1xn} to \ref{fig:4xn} for examples.

\subsection*{{\bf $m \times 1$}}

Note that $\Ca{1}{1}$ is a vertex with no edges, and thus  supports a nullhomotopic tour.   The multigraph $\Ca{2}{1}$ consists of two vertices connected by two edges, and the tour produced by these edges is not nullhomotopic.  By Proposition \ref{prop:oddeven}, cylinders where $m$ is odd and $m > 1$ cannot support a nullhomotopic tour.

The only remaining case is when $m$ is even with $m \geq 4$, and we use induction to produce nullhomotopic tours on these boards.  Consider as our base case the edge path in $\Sc{4}$ shown in Frame (A) of Figure \ref{fig:1xn}.   Suppose that we have a path in $\Sc{m}$ using edge $(m-2, -\frac{m}{2}+1)-(m-1, -\frac{m}{2}+3)$ and whose image in $\Ca{m}{1}$ is a nullhomotopic tour.  Taking this cycle in $\Sc{m+2}$ and replacing $(m-2, -\frac{m}{2}+1)-(m-1, -\frac{m}{2}+3)$ with the edge path $(m-2, -\frac{m}{2}+1)-(m, -\frac{m}{2})-(m+1, -\frac{m}{2}+2)-(m-1, -\frac{m}{2}+3)$ creates a path in $\Sc{m+2}$ that uses the edge $(m, -\frac{m}{2})-(m+1, -\frac{m}{2}+2)$ and whose image in $\Ca{m+2}{1}$ is a nullhomotopic tour, completing the induction.  The result of applying this process to our base case tour is shown in Frame (B) of Figure \ref{fig:1xn}.  Visually, our induction step adds a $\sqrt{5} \times \sqrt{5}$ square to a growing rectangle.

\begin{figure}[t]
\subfloat[]{\fbox{\includegraphics[scale = .25]{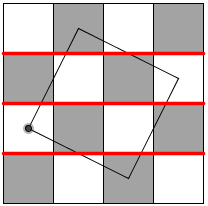}}}
\qquad \qquad \qquad \qquad
\subfloat[]{\fbox{\includegraphics[scale = .25]{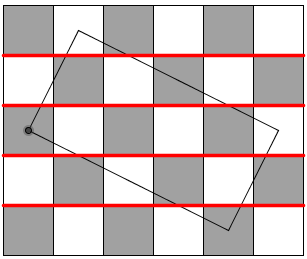}}}
\caption{Lifts of nullhomotopic tours in $\Ca{m}{1}$ for $m=4$ and $6$}
\label{fig:1xn}
\end{figure}

\subsection*{{\bf $1 \times n$}}
By Theorem \ref{thm:watkins}, $\Ca{1}{n}$ does not support a tour when $n>1$.  For the remainder of the section, we assume that $m$ and $n$ are greater than 1.

\subsection*{{\bf $m \times 2$}}
By Theorem \ref{thm:watkins}, $\Ca{2}{2}$ and $\Ca{4}{2}$ cannot support tours.  We produce nullhomotopic tours on $\Ca{m}{2}$ for all values of $m \geq 3$ other than 4 by induction, taking the paths shown in Frames (A) and (C) of Figure \ref{fig:2xn} as our base cases.  Suppose there is a cycle in $\Sc{m}$ that includes the edges $(m-1, \frac{-p}{2}+\frac{5}{2})-(m-2, \frac{-p}{2}+\frac{9}{2})$ and $(m-1, \frac{p}{2}-\frac{3}{2})-(m-2, \frac{p}{2}-\frac{7}{2})$ where $p = m - 1$ if $m$ is even and $p = m$ if $m$ is odd, whose image in $\Ca{m}{2}$ is a nullhomotopic tour.  Since $\Sc{m} \subset \Sc{m+2}$, this constitutes a cycle in $\Sc{m+2}$.  We can replace the edges listed above with two paths consisting of 3 edges each: $(m-1, \frac{-p}{2}+\frac{5}{2})-(m+1, \frac{-p}{2}+\frac{3}{2})-(m, \frac{-p}{2}+\frac{7}{2})-(m-2, \frac{-p}{2}+\frac{9}{2})$ and $(m-1, \frac{p}{2}-\frac{3}{2})-(m+1, \frac{p}{2}-\frac{1}{2})-(m, \frac{p}{2}-\frac{5}{2})-(m-2, \frac{p}{2}-\frac{7}{2})$.  Note that this newly formed path in $\Sc{m+2}$ includes the edges $(m+1, \frac{-p}{2}+\frac{3}{2})-(m, \frac{-p}{2}+\frac{7}{2})$ and $(m+1, \frac{p}{2}-\frac{1}{2})-(m, \frac{p}{2}-\frac{5}{2})$, and the image of this path in $\Ca{m+2}{2}$ is a nullhomotopic tour.  This completes our induction.  Frames (B) and (D) of Figure \ref{fig:2xn} show the results of applying this argument to Frames (A) and (C) respectively.  Visually, this has the effect of adding two parallelograms to the path in $S$.

\begin{figure}[t]
\subfloat[]{\fbox{\includegraphics[scale = .265]{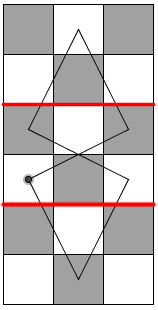}}}
\hfill
\subfloat[]{\fbox{\includegraphics[scale = .2]{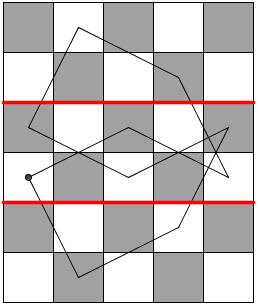}}}
\hfill
\subfloat[]{\fbox{\includegraphics[scale = .2]{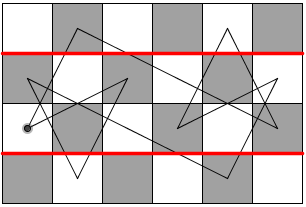}}}
\hfill
\subfloat[]{\fbox{\includegraphics[scale = .2]{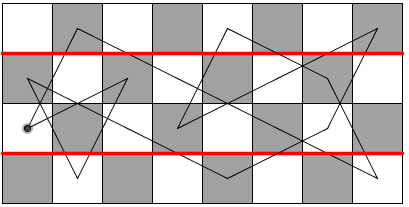}}}
\caption{Lifts of nullhomotopic tours in $\Ca{m}{2}$ for $m=3, 5, 6,$ and $8$}
\label{fig:2xn}
\end{figure}

\subsection*{{\bf $2 \times n$}}
The multigraph $\Ca{2}{n}$ cannot support a tour because there are no simple edge cycles in $\Sc{2}$.  For the remainder of this section, we assume that $m$ and $n$ are greater than 2.

\subsection*{{\bf $m \times 3$}}
By Proposition \ref{prop:oddeven}, $\Ca{m}{3}$ does not support a nullhomotopic tour for odd $m$.  In Frames (A), (B), and (C) of Figure \ref{fig:3xn}, we have paths in $\Sc{4}$, $\Sc{6}$, and $\Sc{8}$ whose respective images in $\Ca{4}{3}$, $\Ca{6}{3}$, and $\Ca{8}{3}$ are nullhomotopic tours. By Theorem \ref{thm:schwenk}, $\Ca{m}{3}$ supports a nullhomotopic tour for even $m$ when $m \geq 10$.

\begin{figure}[t]
\subfloat[]{\fbox{\includegraphics[scale = .2]{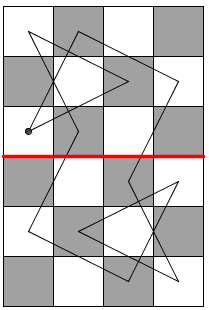}}}
\hfill
\subfloat[]{\fbox{\includegraphics[scale = .15]{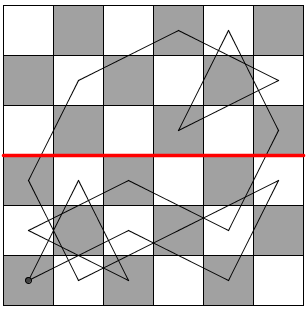}}}
\hfill
\subfloat[]{\fbox{\includegraphics[scale = .15]{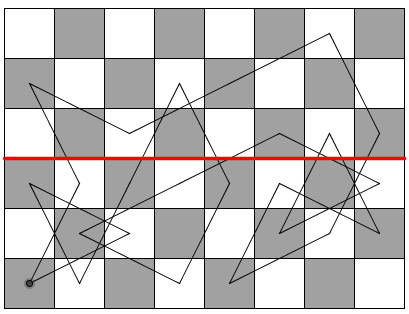}}}
\hfill
\subfloat[]{\fbox{\includegraphics[scale = .26]{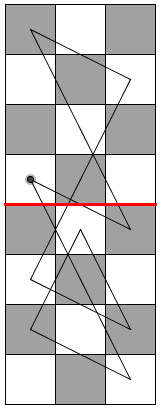}}}
\hfill
\subfloat[]{\fbox{\includegraphics[scale = .29]{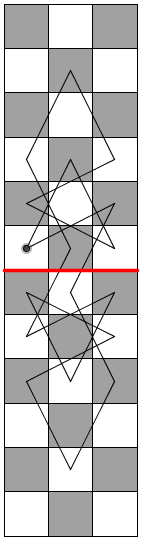}}}
\hfill
\subfloat[]{\fbox{\includegraphics[scale = .29]{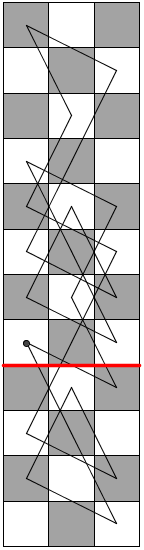}}}
\caption{(A-C) Lifts of nullhomotopic tours in $\Ca{m}{3}$ for $m=4, 6,$ and $8$  (D-F) Lifts of nullhomotopic tours in $\Ca{3}{n}$ for $n=4, 6,$ and $8$}
\label{fig:3xn}
\end{figure}

\subsection*{{\bf $3 \times n$}}
By Proposition \ref{prop:oddeven}, $\Ca{3}{n}$ does not support a nullhomotopic tour for odd values of $n$.  By Frames (D), (E), and (F) of Figure \ref{fig:3xn}, we have paths in $\Sc{3}$ whose respective images in $\Ca{3}{4}$, $\Ca{3}{6}$, and $\Ca{3}{8}$ are nullhomotopic tours.  By Theorem \ref{thm:schwenk}, $\Ca{3}{n}$ supports a nullhomotopic tour for even values of $n$ where $n \geq 10$.  For the rest of this section, we will assume that $m$ and $n$ are greater than $3$.

\subsection*{{\bf $m \times 4$}}
By Theorem \ref{thm:watkins}, the graph $\Ca{4}{4}$ cannot support a nullhomotopic tour.  For all other values of $m > 3$, we use induction to construct a nullhomotopic tour in $\Ca{m}{4}$.  We take the paths in $\Sc{3}$, $\Sc{5}$, and $\Sc{7}$ shown in Frame (D) of Figure \ref{fig:3xn}, and Frames (A) and (C) of Figure \ref{fig:4xn} as our base cases.  Assume there exists a path in $\Sc{m}$ that includes the edge $(m-2, -1)-(m-1, -3)$ and whose image in $\Ca{m}{4}$ is a nullhomotopic tour.   Taking our path to be in $\Sc{m+3}$, we translate the path shown in Frame (D) of Figure \ref{fig:3xn} with the edge $(0,0)-(1,-2)$ removed $m$ units to the right.  We concatenate the translated path with our original by replacing edge $(m-2, -1)-(m-1, -3)$ with the edges $(m-2, -1)~-~(m,~0)$ and $(m-1,~-3)~-~(m+1,~ -2)$.  Note that this new path in $\Sc{m+3}$ includes the edge $(m+1, -1)-(m+2, -3)$ and that its image in $\Ca{m+3}{4}$ is nullhomotopic, completing the induction.  So there exists a path in $\Sc{m}$ for all $m \neq 1, 2, 4$ whose image in $\Ca{m}{4}$ is a nullhomotopic tour.  An example of this process is shown in Frame (B) of Figure \ref{fig:4xn}.
 
\begin{figure}[t]
\subfloat[]{\fbox{\includegraphics[scale = .20]{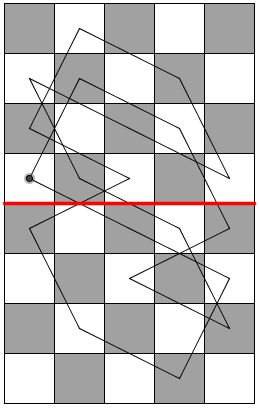}}}
\hfill
\subfloat[]{\fbox{\includegraphics[scale = .20]{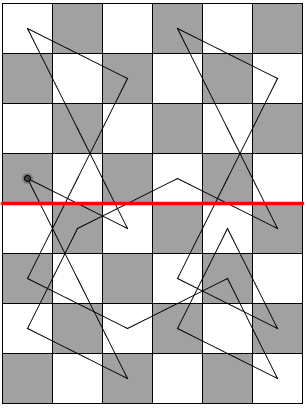}}}
\hfill
\subfloat[]{\fbox{\includegraphics[scale = .20]{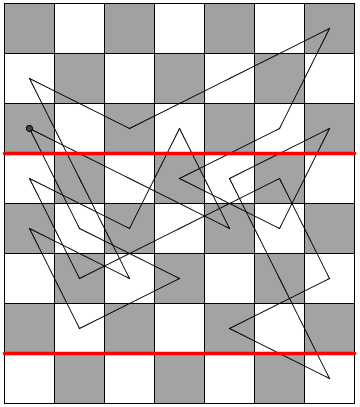}}}
\hfill
\subfloat[]{\fbox{\includegraphics[scale = .24]{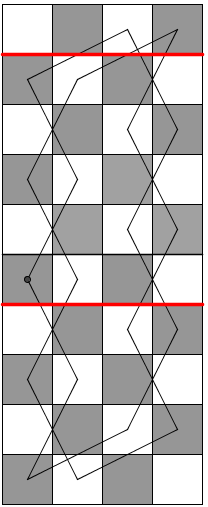}}}
\hfill
\subfloat[]{\fbox{\includegraphics[scale = .26]{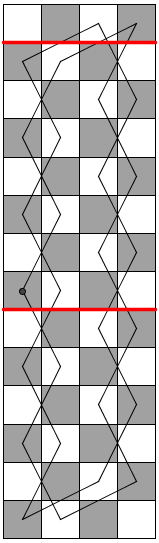}}}
\caption{(A-C) Lifts of nullhomotopic tours in $\Ca{m}{4}$ for $m=5, 6,$ and $7$  (D-E) Lifts of nullhomotopic tours in $\Ca{4}{n}$ for $n=5$ and $7$}
\label{fig:4xn}
\end{figure}

\subsection*{{\bf $4 \times n$}}
By Theorem \ref{thm:watkins}, when $n$ is even, $\Ca{4}{n}$ does not support a nullhomotopic tour.  When $n$ is odd, we construct nullhomotopic tours on $\Ca{4}{n}$ by induction. We take the path in $\Sc{4}$ shown in Frame (D) of Figure \ref{fig:4xn} as our base case.  Suppose there is a path in $\Sc{4}$ that uses the edges $(0,-n+1)-(2,-n+2)$, $(1,-n+1)-(3,-n+2)$, $(0,n-1)-(2,n)$, and $(1,n-1)-(3,n)$, and whose image in $\Ca{4}{n}$ is a nullhomotopic tour.  We replace these edges as follows:
\begin{itemize} 
\item insert the path $(0,-n+1)-(1,-n-1)-(3,-n)-(2,-n+2)$ and delete $(0,-n+1)-(2,-n+2)$,
\item insert the path $(1,-n+1)-(0,-n-1)-(2,-n)-(3,-n+2)$ and delete $(1,-n+1)-(3,-n+2)$,
\item insert the path $(0,n-1)-(1,n+1)-(3,n+2)-(2,n)$ and delete $(0,n-1)-(2,n)$, and
\item insert the path $(1,n-1)-(0,n+1)-(2,n+2)-(3,n)$ and delete $(1,n-1)-(3,n)$.
\end{itemize}
We these edge replacements, this new path in $\Sc{4}$ includes the edges $(0,-n-1)-(2,-n)$, $(1,-n-1)-(3,-n)$, $(0,n+1)-(2,n+2)$, and $(1,n+1)-(3,n+2)$, and the image of this path in $\Ca{4}{n+2}$ is a nullhomotopic tour.  This completes our induction.  An example of this process is shown in Frame (E) of Figure \ref{fig:4xn}.  Visually, this has the effect of adding two squares and two parallelograms, one of each at the top and one of each at the bottom, to our tour.

Shifting now to consider $\Ta{m}{n}$, other than for a few cases of small boards, the parites of $m$ and $n$ completely determine if it is possible to construct a nullhomotopic tour.

\begin{theorem}
When $m + n > 3$, the multigraph $\Ta{m}{n}$ supports a nullhomotopic tour if and only if at least one of $m$ or $n$ is even.
\label{cor:ToriNull}
\end{theorem}

\begin{proof}
Corolloary \ref{prop:oddeven} establishes the forward direction.  For the backward direction, note that $\Ta{m}{n}$ contains $\Ca{m}{n}$ and $\Ca{n}{m}$ as  sub-multigraphs. If $\Ca{m}{n}$ or $\Ca{n}{m}$ support a nullhomotopic tour, then $\Ta{m}{n}$ must as well.  

Combining these observations with Theorem \ref{thm:cylnull}, it suffices to construct nullhomotopic tours on $\Ta{2}{2}$, $\Ta{4}{2}$ and $\Ta{4}{4}$.  We apply statement (3) of Corollary \ref{cor:cover}, and produce edge cycles in $\Pc$ in Frames (A), (B), and (C) of Figure \ref{fig:2xnTnull} whose respective images in $\Ta{2}{2}$, $\Ta{4}{2}$ and $\Ta{4}{4}$ are nullhomotopic tours.
\end{proof}

For small boards, $\Ta{1}{1}$ consists of a vertex and eight edges, and so the constant path is a nullhomotopic tour.  The pseudo-graphs $\Ta{1}{2}$ and $\Ta{2}{1}$ are isomorphic and do not admit tours because there are no bigons in $\Pc$.

\begin{figure}[t]
\subfloat[]{\fbox{\includegraphics[scale = .25]{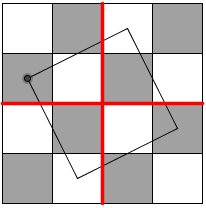}}}
\hfill
\subfloat[]{\fbox{\includegraphics[scale = .25]{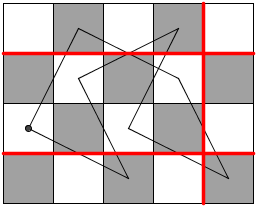}}}
\hfill
\subfloat[]{\fbox{\includegraphics[scale = .25]{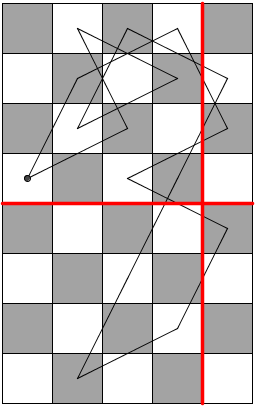}}}
\caption{Lifts of nullhomotopic tours in $\Ta{2}{2}$, $\Ta{4}{2}$, and $\Ta{4}{4}$}
\label{fig:2xnTnull}
\end{figure}

\section{\bf Generating Tours on Cylinders}

In this section, we characterize the values of $m$ and $n$ so that $\Ca{m}{n}$ supports a tour that realizes a generator.  More specifically, we prove:

\begin{theorem}
The graph $\Ca{m}{n}$ supports a tour that realizes a generator if and only if none of the following hold:
\begin{itemize}
\item $m = 1, 2, 4$, or
\item $m$ is even and $n$ is odd.
\end{itemize}
\label{thm:cylgen}
\end{theorem}

As shown by Corollary \ref{cor:cover}, one way to construct a tour on $\Ca{m}{n}$ that realizes a generator is to construct a tour that uses exactly one cylindrical move.  Such tours on cylinders correspond to open tours on regular boards where the starting and ending squares are connected in $\Ca{m}{n}$ by a cylindrical move.  Significant work has been completed studying open tours on $\Ra{m}{n}$, and we use two particular results from that area in the proof of Theorem \ref{thm:cylgen}.

\begin{theorem}[Cannon, Dolan \cite{cannon-dolan}]
The graph $\Ra{m}{n}$ supports an open knight's tour between any two squares of opposite colors if and only if the product $mn$ is even and $m$ and $n$ are both at least 6.
\label{thm:cannondolan}
\end{theorem}

For regular $m \times n$ boards in which $mn$ is odd, it is reasonable to expect there to be open knight's tours beginning and ending at squares that have the same color as the corner squares.

\begin{theorem}[Ralston \cite{ralston}]
If $m$ and $n$ are both odd, both at least 5 with one not equal to 5, then for any pair of squares on the regular $m \times n$ board with same color as the corners, there exists an open knight's tour on $\Ra{m}{n}$ with those as the starting and ending squares.
\label{thm:ralston}
\end{theorem}

In our context, Theorems \ref{thm:cannondolan} and \ref{thm:ralston} combine to demonstrate that:

\begin{corollary}
If $m$ and $n$ are both odd, and greater than or equal to 5 with at least one greater than 5, or $m$ and $n$ are both at least 6 and $n$ is even, then there is a knight's tour on $\Ca{m}{n}$ that realizes a generator.
\label{cor:cylgen}
\end{corollary}

Together, Corollary \ref{cor:cylgen} and Proposition \ref{prop:oddeven} establish Theorem \ref{thm:cylgen} for all $\Ca{m}{n}$ where $m \geq 6$ and $n \geq 5$.  So, we consider $\Ca{m}{n}$ where $m < 6$ or $n < 5$ and proceed by cases.
Our arguments often apply statement (2) from Corollary \ref{cor:cover}.  Specifically, we will produce paths in $\Sc{m}$ starting at $(0,0)$ and ending at $(0,n)$ or $(0,-n)$ whose image in $\Ca{m}{n}$ is a tour.

\subsection*{{\bf $m \times 1$}} 
Because $\Ca{1}{1}$ is a vertex with no edges, it does not support a tour realizing a generator.  By Proposition \ref{prop:oddeven}, $\Ca{m}{1}$ does not support a generating tour for even values of $m$.  For odd values of $m$, we use induction to produce a tour that realizes a generator on $\Ca{m}{1}$.  Consider the edge path in $\Sc{3}$ shown in Frame (A) of Figure \ref{fig:1xnGC} as our base case.  Suppose that there exists a path in $\Sc{m}$ whose image in $\Ca{m}{1}$ is a tour that realizes a generator and that uses the edge $(m-1, \frac{-m+3}{2})-(m-2, \frac{-m-1}{2})$.  Considering this as a path in $\Sc{m+2}$, we replace that edge with the 3-edge path $(m-1, \frac{-m+3}{2})-(m+1, \frac{-m+1}{2})-(m, \frac{-m-3}{2})-(m-2, \frac{-m-1}{2})$.  Note that this path in $\Sc{m+2}$ includes the edge $(m+1, \frac{-m+1}{2})-(m, \frac{-m-3}{2})$ and that the image of this path in $\Ca{m+2}{1}$ is a tour that realizes a generator.  This completes our induction.  An example of this process is shown in Frame (B) of Figure \ref{fig:1xnGC}. Visually, we are adding a $\sqrt{5} \times \sqrt{5}$ square to our path.

\begin{figure}[t]
\subfloat[]{\fbox{\includegraphics[scale = .27]{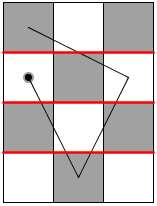}}}
\hfill
\subfloat[]{\fbox{\includegraphics[scale = .20]{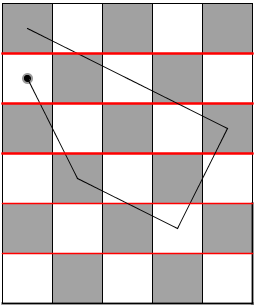}}}
\hfill
\subfloat[]{\fbox{\includegraphics[scale = .27]{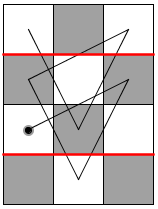}}}
\hfill
\subfloat[]{\fbox{\includegraphics[scale = .25]{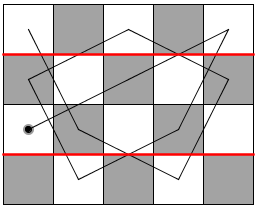}}}
\hfill
\subfloat[]{\fbox{\includegraphics[scale = .20]{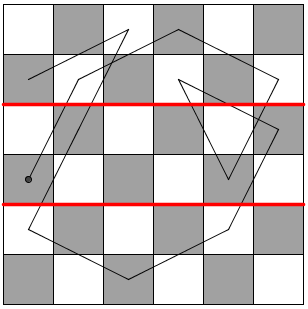}}}
\caption{(A-B) Lifts of tours in $\Ca{m}{1}$ realizing a generator for $m=3$ and $5$  (C-E) Lifts of tours in $\Ca{m}{2}$ realizing a generator for $m=3, 5,$ and $6$}
\label{fig:1xnGC}
\end{figure}

\subsection*{{\bf $1 \times n$}}
When $n>1$, Theorem \ref{thm:watkins} shows that $\Ca{1}{n}$ does not support a tour.  For the remainder of this section, we assume that $m$ and $n$ are greater than 1.

\subsection*{{\bf $m \times 2$}} 
By Theorem \ref{thm:watkins}, $\Ca{m}{2}$ does not support a tour when $m=2$ or $4$.  For all other values of $m$, we use induction to construct a tour on $\Ca{m}{2}$ that realizes a generator.  Consider the edge paths in $\Sc{3}$ shown in Frames (C) and (E) of Figure \ref{fig:1xnGC} as our base cases.  Suppose that there is a path in $\Sc{m}$ that includes the edges $(m-1, 2)-(m-2, 0)$ and $(m-1, 1)-(m-2, -1)$ whose image in $\Ca{m}{2}$ is a tour that realizes a generator.  Taking this as a path in $\Sc{m+2}$, we replace the edge $(m-1, 2)-(m-2, 0)$ with the path $(m-1, 2)-(m+1, 1)-(m, -1)-(m-2,0)$.  Additionally, we replace $(m-1,~1)~-~(m-2,~-1)$ with $(m-1, 1)-(m+1, 2)-(m, 0)-(m-2, -1)$.  Note that this new path in $\Sc{m+2}$ includes the edges $(m+1, 2)-(m, 0)$ and $(m+1, 1)-(m, -1)$ and that the image of this path in $\Ca{m+2}{2}$ is a tour that realizes a generator, completing our induction.  An example of this process is shown in Frame (D)  of Figure \ref{fig:1xnGC}. Visually we have added a square and a parallelogram to our path.

\subsection*{{\bf $2 \times n$}}
By Proposition \ref{prop:oddeven}, $\Ca{2}{n}$ cannot realize a generator for odd values of $n$.  By Theorem \ref{thm:watkins}, $\Ca{2}{n}$ cannot support a tour when $n$ is even.  For the rest of this section, we assume $m$ and $n$ are greater than 2.

\subsection*{{\bf $m \times 3$}} 
By Proposition \ref{prop:oddeven}, $\Ca{m}{3}$ does not support a tour that realizes a generator when $m$ is even.  When $m$ is odd and greater than 3, we inductively show that $\Ca{m}{3}$ supports a tour that realizes a generator using the paths in $\Sc{3}$ and $\Sc{5}$ shown in Frames (A) and (B) of Figure \ref{fig:3xnGC} as our base cases.  Assume there exists a path in $\Sc{m}$ that uses the edge $(m-1, 1)-(m-2, 3)$ and whose image in $\Ca{m}{3}$ realizes a generator. Considering this as a path in $\Sc{m+4}$ we add the path shown in Frame (D) of Figure \ref{fig:3xnGC} by placing this $4 \times 3$ board on the rectangle given by corners $(m, 1), (m,3), (m+3,3),$ and $(m+3,1)$.   Deleting $(m-1, 1)-(m-2, 3)$ and adding edges $(m-1,1)-(m, 3)$ and $(m-2, 3)-(m, 2)$ produces a path in $\Sc{m+4}$.  This path includes the edge $(m+3, 1)-(m+2, 3)$ and its image in $\Ca{m+4}{3}$ is a tour that realizes a generator, completing our induction.  An example of this process is shown in Frame (C) of Figure \ref{fig:3xnGC}.

\begin{figure}[t]
\subfloat[]{\fbox{\includegraphics[scale = .27]{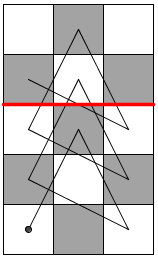}}}
\hfill
\subfloat[]{\fbox{\includegraphics[scale = .21]{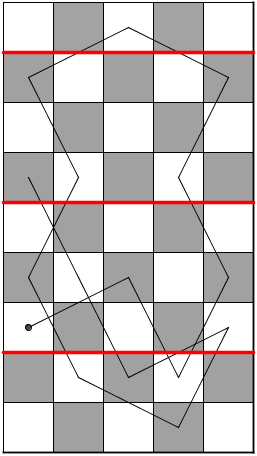}}}
\hfill
\subfloat[]{\fbox{\includegraphics[scale = .21]{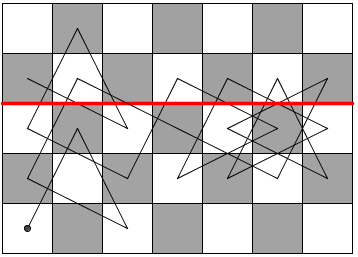}}}
\hfill
\subfloat[]{\fbox{\includegraphics[scale = .25]{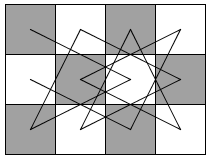}}}
\caption{(A-C) Lifts of tours in $\Ca{m}{3}$ realizing a generator for $m=3, 5,$ and $7$  (D) Path on $\Ra{4}{3}$ used in the induction step of the $m \times 3$ case}
\label{fig:3xnGC}
\end{figure}

\subsection*{{\bf $3 \times n$}}
Paths in $\Sc{3}$ whose images in $\Ca{3}{5}$ and $\Ca{3}{6}$ realize a generator are shown in Frames (B) and (C) of Figure \ref{fig:mx3CG} respectively.  For all values of $n$ except 5 and 6, we will inductively construct a tour in $\Ca{3}{n}$ that realizes a generator.  Consider the paths in $\Sc{3}$ in Frames (A), (D), (E), and (F) of Figure \ref{fig:mx3CG} whose respective images in $\Ca{3}{4}$, $\Ca{3}{7}$, $\Ca{3}{9}$, and $\Ca{3}{10}$ are tours that realize a generator as our base cases.  Suppose there exists a path in $\Sc{3}$ that begins at $(0,0)$, ends at $(0, -n)$, whose image in $\Ca{3}{n}$ is a tour that realizes a generator, and whose image in $\Ca{3}{n+4}$ is not incident to non-base point vertices $(a,b)$ for $b \leq 3$.  We concatenate this path with the result of translating the path shown in Frame (A) of Figure \ref{fig:mx3CG} downward by $n$.  The result is a path in $\Sc{3}$ that begins at $(0,0)$, ends at $(0, -n-4)$, whose image in $\Ca{3}{n+4}$ is a tour that realizes a generator, and whose image in $\Ca{3}{n+8}$ is not incident to non-base point vertices $(a,b)$ for $b \leq 3$.  This completes our induction.  Frame (G) of Figure \ref{fig:mx3CG} shows the result of applying this process to Frame (D).

\begin{figure}[t]
\subfloat[]{\fbox{\includegraphics[scale = .27]{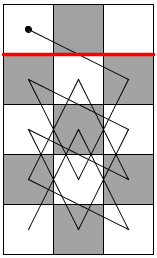}}}
\hfill
\subfloat[]{\fbox{\includegraphics[scale = .27]{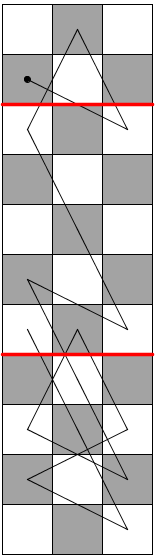}}}
\hfill
\subfloat[]{\fbox{\includegraphics[scale = .27]{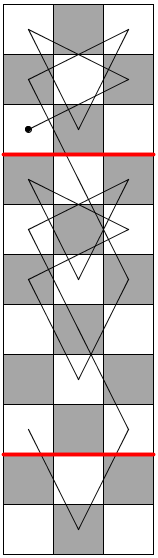}}}
\hfill
\subfloat[]{\fbox{\includegraphics[scale = .28]{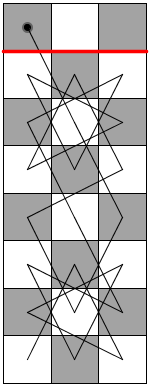}}}\\
\subfloat[]{\fbox{\includegraphics[scale = .27]{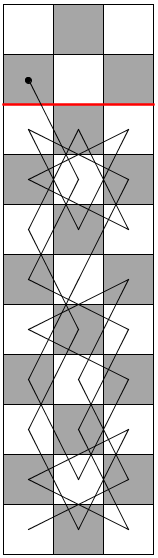}}}
\qquad \qquad
\subfloat[]{\fbox{\includegraphics[scale = .27]{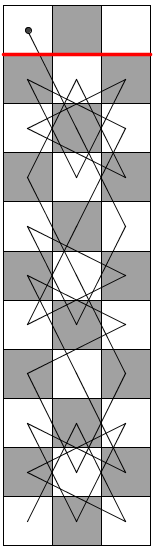}}}
\qquad \qquad
\subfloat[]{\fbox{\includegraphics[scale = .29]{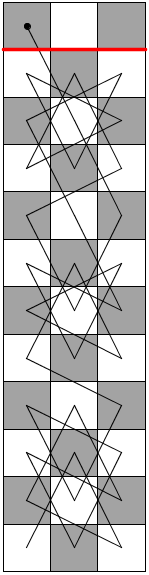}}}
\caption{Lifts of tours in $\Ca{3}{n}$ realizing a generator for $n=4, 5, 6, 7, 9, 10,$ and $11$}
\label{fig:mx3CG}
\end{figure}

\subsection*{{\bf $m \times 4$}} 
By Theorem \ref{thm:watkins}, $\Ca{4}{4}$ cannot support a tour realizing a generator.  Frames (A) through (C) of Figure \ref{fig:4xnCG} show paths in $\Sc{5}$, $\Sc{6}$, and $\Sc{7}$ whose respective images in $\Ca{5}{4}$, $\Ca{6}{4}$, and $\Ca{7}{4}$ are tours that each realize a generator.  For each $m \geq 5$, we use induction to construct tours on $\Ca{m}{4}$ taking the paths on $\Ca{5}{4}$, $\Ca{6}{4}$, and $\Ca{7}{4}$ as our base cases.  Assume there is a path in $\Sc{m}$ that uses the edges $(m-1, 0)-(m-2, 2)$ and $(m-2, 0)-(m-1, 2)$ whose image in $\Ca{m}{4}$ realizes a generator.  Taking our path to be in $\Sc{m+3}$, we place the $3 \times 4$ board shown in Frame (E) of Figure \ref{fig:4xnCG} on the rectangle given by the corners $(m, 0), (m,3), (m+2,3),$ and $(m+2,0)$.   Deleting the edges $(m-1, 0)-(m-2, 2)$ and $(m-2, 0)-(m-1, 2)$ and adding the edges $(m-2, 0)-(m, 1)$, $(m-1, 0)-(m+1, 1)$, $(m-2, 2)-(m, 3)$, and $(m-1, 2)-(m+1, 3)$ produces a path that includes the edges $(m+2, 0)-(m+1, 2)$ and $(m+1, 0)-(m+2, 2)$ and whose image in $\Ca{m+3}{4}$ is a tour that realizes a generator, completing the induction.  The result of applying this process to Frame (B) of Figure \ref{fig:4xnCG} is shown in Frame (D).
\begin{figure}[t]
\subfloat[]{\fbox{\includegraphics[scale = .15]{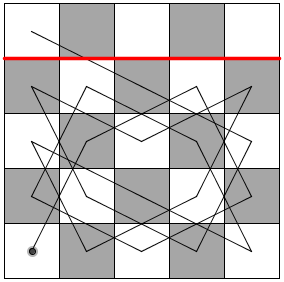}}}
\hfill
\subfloat[]{\fbox{\includegraphics[scale = .17]{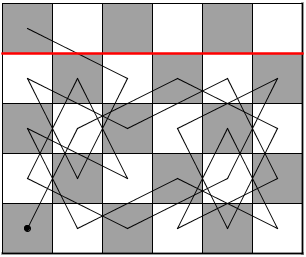}}}
\hfill
\subfloat[]{\fbox{\includegraphics[scale = .17]{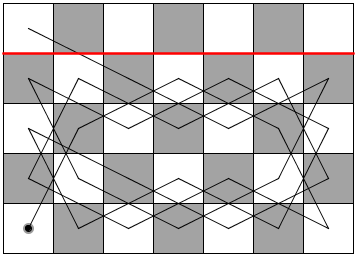}}}
\hfill
\subfloat[]{\fbox{\includegraphics[scale = .15]{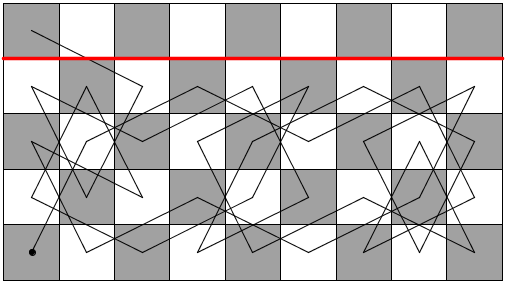}}}
\hfill
\subfloat[]{\fbox{\includegraphics[scale = .27]{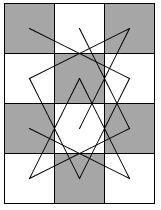}}}
\caption{(A-D) Lifts of tours in $\Ca{m}{4}$ realizing a generator for $m=5, 6, 7$, and $9$ (E) Path used in induction step of $m \times 4$ case}
\label{fig:4xnCG}
\end{figure}

\subsection*{{\bf $4 \times n$}}
$\Ca{4}{n}$ does not support a tour for  odd nor even $n$ by Proposition \ref{prop:oddeven} and Theorem \ref{thm:watkins} respectively. 

\subsection*{{\bf $5 \times n$}}
Frame (A) of Figure \ref{fig:mx5CG} shows that there exists a tour on $\Ca{5}{5}$ that realizes a generator.  For all odd values of $n$ with $n \geq 7$, there exists a tour that realizes a generator on $\Ca{5}{n}$ by Corollary \ref{cor:cylgen}.  For all even values of $n$ with $n \geq 6$ we will inductively show that $\Ca{5}{n}$ supports a tour that realizes a generator.  We use the paths in $\Sc{5}$ shown in Frame (A) of Figure \ref{fig:4xnCG} and Frame (B) of Figure \ref{fig:mx5CG} as our base cases.  Suppose we have a path in $\Sc{5}$ that begins at $(0,0)$, ends at $(0,-n)$, whose image in $\Ca{5}{n}$ realizes a generator, and whose image in $\Ca{5}{n+4}$ is not incident to non-base point vertices $(a,b)$ with $b \leq 3$.  Concatenating this path with the result of translating Frame (A) of Figure \ref{fig:4xnCG}  downward by $n$ gives a new path that ends at $(0,-n-4)$.  The image of this path in $\Ca{5}{n+4}$ realizes a generator, and the image in $\Ca{5}{n+8}$ is not incident to non-base point vertices $(a,b)$ with $b \leq 3$.  The result of applying this process to Frame (A) of Figure \ref{fig:4xnCG} is shown in Frame (C) of Figure \ref{fig:mx5CG}.

\begin{figure}[t]
\subfloat[]{\fbox{\includegraphics[scale = .25]{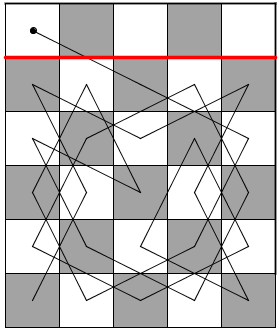}}}
\hfill
\subfloat[]{\fbox{\includegraphics[scale = .25]{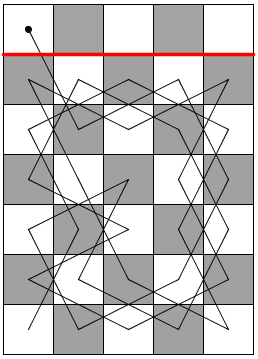}}}
\hfill
\subfloat[]{\fbox{\includegraphics[scale = .25]{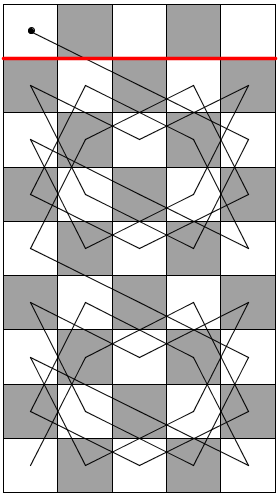}}}
\caption{Lifts of tours in $\Ca{5}{n}$ realizing a generator for $n=5, 6,$ and $8$}
\label{fig:mx5CG}
\end{figure}

\section{\bf Generating Tours on Tori}

The main result of this section is Theorem \ref{thm:torusgen} which extends Theorem \ref{thm:cylgen}.  Specifically, we characterize all $m$ and $n$ so that $\Ta{m}{n}$ admits a tour that realizes a longitude.

\begin{theorem}
When $m$ and $n$ are not both 1,  $\Ta{m}{n}$ supports a tour that realizes a longitude if and only if $m$ is odd or $n$ is even.
\label{thm:torusgen}
\end{theorem}

Note that if $\Ca{m}{n}$ admits a tour that realizes a generator, then $\Ta{m}{n}$ admits a tour that realizes a longitude.  By Theorem \ref{thm:cylgen}, it suffices to show that $\Ta{m}{n}$ admits a tour that realizes a longitude in the following cases: (i) $m=1$ and $n>1$, (ii) $m=2$ and $n$ is even, and (iii) $m=4$ and $n$ is even.  We apply statement (4) of Corollary \ref{cor:cover} in these arguments.  That is, we construct a tour in $\Ta{m}{n}$ that realizes a longitude by constructing a path in $\Pc$ that begins at $(0, 0)$ and ends at $(0, n)$ or $(0, -n)$ whose image in $\Ta{m}{n}$ is a tour.

\begin{figure}[t]
\subfloat[]{\fbox{\includegraphics[scale = .27]{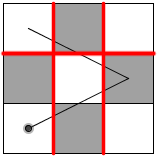}}}
\hfill
\subfloat[]{\fbox{\includegraphics[scale = .27]{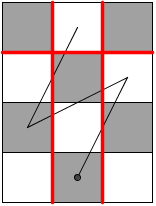}}}
\hfill
\subfloat[]{\fbox{\includegraphics[scale = .27]{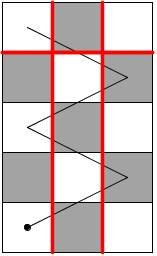}}}
\hfill
\subfloat[]{\fbox{\includegraphics[scale = .25]{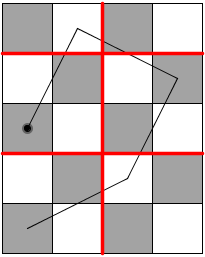}}}
\hfill
\subfloat[]{\fbox{\includegraphics[scale = .25]{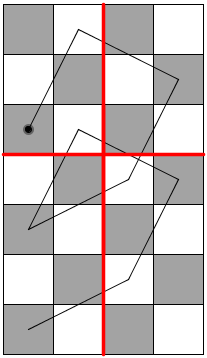}}}
\caption{(A-C) Lifts of tours realizing a generator in $\Ta{1}{n}$ for $n=2, 3,$ and $4$ (D-E) Lifts of tours realizing a generator in $\Ta{2}{n}$ for $n = 2$ and $4$}
\label{fig:mx1GT}
\end{figure}

\subsection*{{\bf $1 \times n$}}
We use induction to show that $\Ta{1}{n}$ supports a tour realizing a longitude for all values of $n > 1$ using the boards in Frames (A) and (B) of Figure \ref{fig:mx1GT} as our base cases.  Suppose there is a path beginning at $(0,0)$, ending at $(0,n)$ in $\Pc$, whose image in $\Ta{1}{n}$ realizes a longitude, and whose image in $\Ta{1}{n+2}$ is not incident to any $(a,b)$ with $b \geq n+1$.  Concatenating this path with $(0,n)-(2,n+1)-(0,n+2)$ gives a path in $\Pc$ beginning at $(0,0)$, ending at $(0,n+2)$, whose image in $\Ta{1}{n+2}$ realizes a longitude, and whose image in $\Ta{1}{n+4}$ is not incident to any $(a,b)$ with $b \geq n+3$.  Frame (C) of Figure \ref{fig:mx1GT} shows the result of applying this process to Frame (A).

\subsection*{{\bf $2 \times n$}}
We use induction to show that $\Ta{2}{n}$ supports a tour that realizes a longitude when $n$ is even, using the tour on $\Ta{2}{2}$ shown in Frame (D) of Figure \ref{fig:mx1GT} as our base case.  Suppose we have a path in $\Pc$ that begins at $(0,0)$, ends at $(0,-n)$, whose image in $\Ta{2}{n}$ realizes a longitude, and whose image in $\Ta{2}{n+2}$ is not incident to $(0,1), (1,3),$ and $(1,4)$.  Concatenating this path in $\Pc$ with $(0,-n)-(1,-n+2)-(3,-n+1)-(2,-n-1)-~(0,~-n~-~2)$ yields a path that begins at $(0,0)$, ends at $(0,-n-2)$, whose image in $\Ta{2}{n+2}$ realizes a longitude, and whose image in $\Ta{2}{n+4}$ is never incident to $(0,1), (1,3),$ and $(1,4)$.  Frame (E) of Figure \ref{fig:mx1GT} shows the result of applying this process to Frame (D).

\subsection*{{\bf $4 \times n$}}
The multigraph $\Ta{4}{2}$ supports a tour realizing a longitude as shown by the path in $\Pc$ in Frame (A) of Figure \ref{fig:mx4GT}.  For all even values of $n \geq 4$, we use induction to show that $\Ta{4}{n}$ supports a tour that realizes a longitude with Frames (B) and (C) as our base cases.  Suppose we have a path in $\Pc$ that begins at $(0,0)$, ends at $(0,-n)$, whose image in $\Ta{4}{n}$ realizes a longitude, and whose image in $\Ta{4}{n+4}$ is never incident to non-base point vertices $(a,b)$ where $b \leq 3$.  We concatenate the path in $\Pc$ with the image of Frame (B) of Figure \ref{fig:mx4GT} under downward translation by $n$.  This yields a path in $\Pc$ that begins at $(0,0)$,  ends at $(0,-n-4)$, whose image in $\Ta{4}{n+4}$ realizes a longitude, and whose image in $\Ta{4}{n+8}$ is never incident to non-base point vertices $(a,b)$ where $b \leq 3$.  Frame (D) of Figure \ref{fig:mx4GT} shows the result of applying this process to Frame (B).

\begin{figure}[t]
\subfloat[]{\fbox{\includegraphics[scale = .25]{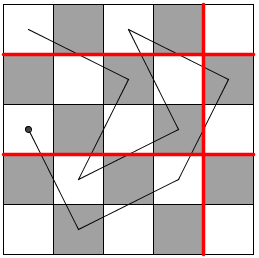}}}
\hfill
\subfloat[]{\fbox{\includegraphics[scale = .25]{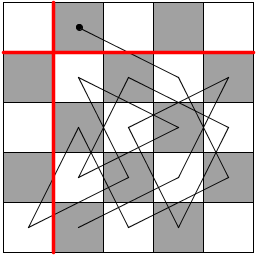}}}
\hfill
\subfloat[]{\fbox{\includegraphics[scale = .2]{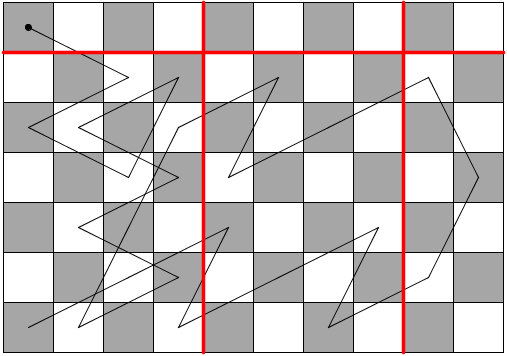}}}
\hfill
\subfloat[]{\fbox{\includegraphics[scale = .2]{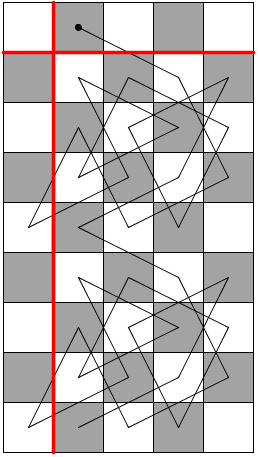}}}
\caption{Lifts of tours realizing a generator in $\Ta{4}{n}$ for $n=2, 4, 6$, and $8$} 
\label{fig:mx4GT}
\end{figure}

\section{\bf Future Work}

This work constitutes the first step in understanding the topology of knight's tours on surfaces.  We are interested in the following general question: given a surface $S$ with dimensions $m \times n$ and an element $\lambda \in~\pi_1(S)$, is there a knight's tour on $S$ that realizes $\lambda$?  Theorems \ref{thm:cylnull}, \ref{cor:ToriNull}, \ref{thm:cylgen}, and \ref{thm:torusgen} settle some of the simplest cases of this question, specifically when $S$ is a cylinder or torus and $\lambda$ is the identity, a generator of the fundamental group the cylinder, or the homotopy class of a longitude in a torus.

One can extend this work by considering other surfaces.  The topology of knight's tours on the M\"obius strip, the Klein bottle, and the real projective plane have not been studied.  Another avenue of extension is to characterize, for a fixed surface $S$, which elements of $\pi_1(S)$ can be realized by knight's tours.  Little work has been done on the topology of knight's tours; it is a topic with many open and interesting questions.

\section*{\bf Acknowledgements}

The second author gratefully acknowledges the support of Stockton University's Fellowship for Distinguished Students which made this work possible. 
\bibliographystyle{plain}

\begin{thebibliography}{9} 

\smallskip

\bibitem{cannon-dolan}
Cannon, Dolan, {\it The Knight's Tour}, The Mathematical Gazette {\bf 70} (1986), 91-100.

\smallskip

\bibitem{chia-ong}
Chia, Ong, {\it Generalized Knight's Tour on Rectangular Chessboards}, J. Discrete Applied Mathematics {\bf 150} (2005), 80-98.

\smallskip

\bibitem{cull-decurtins}
Cull, De Curtins, {\it Knight's Tour Revisited}, Fibonacci Quarterly (1978), 276-285.

\smallskip

\bibitem{demaio}
DeMaio, {\it Which Chessboards Have a Closed Knight's Tour on a Cube?}, Electronic J. Combinatorics {\bf 14} (2007), 1-8.

\smallskip

\bibitem{demaio-mathew}
DeMaio, {\it Which Chessboards have a Closed Knight's Tour within the Rectangular Prisim}, Electronic J. Combinatorics {\bf 18:1} (2011), 2-12.

\smallskip

\bibitem{erde}
Erde, {\it The Closed Knight Tour Problem in Higher Dimensions}, Electronic J. Combinatorics {\bf 19:4} (2012), 2-16.

\smallskip

\bibitem{erde-golenia-golenia}
Erde, Golenia, Golenia {\it The Closed Knight Tour Problem in Higher Dimensions}, Electronic J. Combinatorics {\bf 21:1} (2010).

\smallskip

\bibitem{lin-wei}
Lin, Wei {\it Optimal algorithms for constructing knight's tours on arbitrary $n \times m$ chessboards}, Discrete Applied Mathematics {\bf 146} (2006), 219-232.

\smallskip

\bibitem{miller-farnsworth}
Miller, Farnsworth, {\it Knight's tours on cylindrical and toroidal boards with one square removed}, J. Discrete Mathematics {\bf 3} (2013), 56-59.

\smallskip

\bibitem{ralston}
Ralston, {\it Knight's Tours on Odd Boards}, J. Recreational Math {\bf 28:3} (1996), 194-200.

\smallskip

\bibitem{schwenk}
Schwenk, {\it Which Rectangular Chessboards have a Knight's Tour}, Mathematics Magazine {\bf 64:5} (1991), 325-332.
    
\smallskip    
    
\bibitem{watkins}
Watkins, {\it Across the Board: The Mathematics of Chessboard Problems}. Princeton: Princeton University Press, 2004, 65-77.

\smallskip    
    
\bibitem{watkins-hoenigman}
Watkins, Hoenigman, {\it Knight's Tour on a Torus} Mathematics Magazine {\bf 70:3} (1997), 175-184.

\smallskip 

\bibitem{yang-zhu-jiang-huang}
Yang, Zhu, Jiang, Huang, {\it Generalized knight's tour on 3D chessboards}, Discrete Applied Mathematics {\bf 158:16} (2010), 1727-1731.

\smallskip 


        
\end{thebibliography}

\end{document}